\documentclass[12pt]{amsart}

\usepackage[utf8]{inputenc}
\usepackage{microtype}
\usepackage{amsfonts}
\usepackage{amsmath}
\usepackage{graphicx}
\usepackage{float}
\usepackage[dvipsnames]{xcolor}
\usepackage{hyperref}
\usepackage{tikz-cd}

\usepackage{amssymb}
\usepackage{enumitem}
\usepackage{mathrsfs}
\usepackage{MnSymbol}

\usepackage{tikz}
\usetikzlibrary{topaths}
\usetikzlibrary{calc}

\usepackage{a4wide}

\usepackage{empheq}

\newtheorem{theorem}{Theorem}[section]
\newtheorem*{theorem*}{Theorem}
\newtheorem{lemma}[theorem]{Lemma}
\newtheorem{proposition}[theorem]{Proposition}
\newtheorem*{proposition*}{Proposition}
\newtheorem{corollary}[theorem]{Corollary}
\newtheorem*{corollary*}{Corollary}

\newtheorem*{conjecture*}{Conjecture}
\newtheorem{question}[theorem]{Question}
\newtheorem*{question*}{Question}

\theoremstyle{definition}
\newtheorem{definition}[theorem]{Definition}
\newtheorem*{definition*}{Definition}
\newtheorem{remark}[theorem]{Remark}
\newtheorem{example}[theorem]{Example}

\newcommand{\N}{\mathbb{N}}
\newcommand{\Z}{\mathbb{Z}}

\usepackage[normalem]{ulem}

\DeclareMathOperator{\Hom}{Hom}
\DeclareMathOperator{\Aut}{Aut}

\DeclareMathOperator{\Out}{Out}
\DeclareMathOperator{\Inn}{Inn}

\DeclareMathOperator{\F}{\mathcal{F}}
\DeclareMathOperator{\FP}{\mathcal{FP}}

\DeclareMathOperator{\Stab}{Stab}

\newcommand{\cdq}{\mathrm{cd}_{\mathbb{Q}}}

\newcommand{\HF}{\mathbf{H}\mathfrak{F}}
\newcommand{\grwr}{\nwsebipropto}

\numberwithin{equation}{section}

\begin{document}

\title[Finiteness properties of stabilisers of oligomorphic actions]{Finiteness properties of stabilisers \\ of oligomorphic actions}
\date{\today}


\author[F.~Fournier-Facio]{Francesco Fournier-Facio}
\address{Department of Pure Mathematics and Mathematical Statistics, University of Cambridge, UK}
\email{ff373@cam.ac.uk}

\author[P.~H.~Kropholler]{Peter H. Kropholler}
\address{School of Mathematical Sciences, University of Southampton, Southampton SO17 1BJ United Kingdom}
\email{p.h.kropholler@southampton.ac.uk}

\author[R.~A.~Lyman]{Robert Alonzo Lyman}
\address{Department of Mathematics and Computer Science, Rutgers University Newark, Newark NJ}
\email{robbie.lyman@rutgers.edu}

\author[M.~C.~B.~Zaremsky]{Matthew C.~B.~Zaremsky}
\address{Department of Mathematics and Statistics, University at Albany (SUNY), Albany, NY}
\email{mzaremsky@albany.edu}

\begin{abstract}
An action of a group on a set is oligomorphic if it has finitely many orbits of $n$-element subsets for all $n$. We prove that for a large class of groups (including all groups of finite virtual cohomological dimension and all countable linear groups), for any oligomorphic action of such a group on an infinite set there exists a finite subset whose stabiliser is not of type $\FP_\infty$. This leads to obstructions on finiteness properties for permutational wreath products and twisted Brin--Thompson groups. We also prove a version for actions on flag complexes, and discuss connections to the Boone--Higman conjecture. In the appendix, we improve on the criterion of Bartholdi--Cornulier--Kochloukova for finiteness properties of wreath products, and the criterion of Kropholler--Martino for finiteness properties of graph-wreath products.
\end{abstract}

\maketitle
\thispagestyle{empty}

\section{Introduction}

Let $G$ be a group acting on an infinite set $S$. Following Cameron \cite{cameron90}, we say that the action is \emph{oligomorphic} if there are finitely many orbits of $n$-element subsets for all $n$, or equivalently finitely many orbits of $n$-tuples of elements for all $n$. In this note we establish a restriction on the possible finiteness properties of stabilisers in $G$ of finite subsets of $S$ under such actions. 

Recall that Kropholler's hierarchy $\HF$ \cite{kropholler:HF} is the smallest class of groups that contains finite groups and has the property that if $G$ admits a cell-permuting action on a finite-dimensional contractible CW-complex with stabilisers in $\HF$, then $G$ is in $\HF$. This is a large class, which includes all groups of finite virtual cohomological dimension, of any cardinality, and all soluble and linear groups of cardinality less than or equal to $\aleph_\omega$. Moreover $\HF$ is subgroup closed, extension closed, and closed under the formation of fundamental groups of graphs of groups.

Our main result is the following.

\begin{theorem}
\label{intro thm action}
    Let $G$ be a group in $\HF$ acting on an infinite set $S$. If the action is oligomorphic, then there exists a non-empty finite subset of $S$ whose stabiliser in $G$ is not of type $\FP_\infty$.
\end{theorem}

In fact, the set in the conclusion of Theorem~\ref{intro thm action} can be taken to be arbitrarily large (Corollary~\ref{cor large sets}).
Recall that a group $G$ is of \emph{type $\FP_n$} if the trivial $G$-module $\Z$ has a projective resolution by $G$-modules that is finitely generated in dimensions up to $n$. For example $G$ is of type $\FP_1$ if and only if it is finitely generated,  and being of type $\FP_2$ is implied by finite presentability, but not vice versa \cite{bestvina97}. A group is of \emph{type $\FP_\infty$} if it is of type $\FP_n$ for all $n$. These properties are commensurability invariants \cite{alonso94}, hence talking about \emph{pointwise} or \emph{setwise} stabilisers of finite sets is equivalent.

Theorem~\ref{intro thm action} is quite surprising, since often strong finiteness properties for a group go hand in hand with strong transitivity properties (like oligomorphicity) together with strong finiteness properties of stabilisers. For example, every group acting oligomorphically on a set with stabilisers of non-empty finite subsets all of type $\FP_\infty$ is itself of type $\FP_\infty$ (see Lemma~\ref{lem:converse}). Thus in some sense Theorem~\ref{intro thm action} is a negation of the converse of this fact.

Also notice that Theorem~\ref{intro thm action} is specific to type $\FP_\infty$---there do exist groups of type $\FP_n$ in $\HF$ acting oligomorphically on infinite sets with stabilisers of finite subsets all of type $\FP_n$, e.g., the Houghton group $H_{n+1}$; see the proof of \cite[Corollary~G]{belk22}. The prototypical example of a group not in $\HF$ is Thompson's group $F$, and in fact $F$ acts oligomorphically on the dyadics in $(0, 1)$, with all stabilisers of type $\FP_\infty$. In Subsection~\ref{subsec question} we discuss this example further, and speculate about the existence of examples of a different nature (Question~\ref{question}).

\medskip

\textbf{Permutational wreath products and twisted Brin--Thompson groups.} We also obtain consequences for the finiteness properties of some group-theoretic constructions that take as input a group from $\HF$. This was in fact the initial motivation for this note. These include permutational wreath products, and also the twisted Brin--Thompson groups. Recall that the \emph{permutational wreath product} $B \wr_S G$, for $B$ a \emph{base} group and $S$ a set with an action of the group $G$, is the semidirect product $(\bigoplus\limits_S B)\rtimes G$, where $G$ shuffles the copies of $B$ via its action on $S$. Twisted Brin--Thompson groups were introduced by Belk and Zaremsky in \cite{belk22} following the Brin--Thompson groups in \cite{brin04}. We will need almost no details about them in this paper, but let us set up some notation and a little intuition for the groups. Given any group $G$ acting faithfully on a set $S$ we get a twisted Brin--Thompson group denoted $SV_G$. This is the group of all homeomorphisms of the product space $C^S$, with a copy of the Cantor set $C$ for each point in $S$, that ``locally'' resemble some combination of copies of Thompson's group $V$ acting on the individual Cantor sets together with the copy of $G$ permuting the Cantor sets via its action on $S$. See \cite{belk22,zar_taste} for far more detail.

\begin{theorem}
\label{intro thm constructions}
    Let $G$ be a group in $\HF$ acting on an infinite set $S$. Then:
    \begin{itemize}
        \item If $B$ is an infinite group, then the permutational wreath product $B \wr_S G$ is not of type $\FP_\infty$.
        \item The twisted Brin--Thompson group $SV_G$ (which exists only when the action is faithful) is not of type $\FP_\infty$.
    \end{itemize}
\end{theorem}

This result is also rather surprising, since the groups $G$ and $B$ themselves may have extremely nice finiteness properties much stronger than type $\FP_\infty$: for instance $B$ could just be the infinite cyclic group $\Z$, and $G$ could admit a finite classifying space. Moreover, there is by now a wealth of literature involving Thompson-like groups that are of type $\FP_\infty$, so the fact that $SV_G$ fails to be $\FP_\infty$ with only the assumptions that $S$ is infinite and $G$ is in $\HF$ is striking.

In the literature on finiteness properties of wreath products, the base is typically assumed to have infinite abelianisation in order for the proofs to work; see \cite{bartholdi15, kropholler:martino}. It turns out that it suffices to just assume that the base is infinite, and we cover this in Appendix~\ref{appendix}. In particular this is what allows us to avoid assuming that $B$ has infinite abelianisation in Theorem~\ref{intro thm constructions}.

\medskip

\textbf{Stronger hypotheses, stronger conclusions.} We also prove a variant of Theorem~\ref{intro thm action}, where we impose some additional restrictions on $G$ and get a stronger conclusion. Recall that a group has the \emph{Howson property} if the intersection of any two finitely generated subgroups is finitely generated \cite{howson54}. Also, let us say that a group is \emph{homologically supercoherent} if every finitely generated subgroup is of type $\FP_\infty$.

\begin{corollary}
\label{intro cor howson}
    Let $G$ be a group in $\HF$ acting on an infinite set $S$. Suppose $G$ is finitely generated, homologically supercoherent, and has the Howson property. If the action is oligomorphic, then there exists a point in $S$ whose stabiliser is not finitely generated.
\end{corollary}

This conclusion is stronger than the conclusion of Theorem~\ref{intro thm action} in that the finite subset in question is a single point, and also that its stabiliser is not even finitely generated. These stronger hypotheses also lead to a stronger conclusion about wreath products and twisted Brin--Thompson groups; see Corollary~\ref{cor:howson_FP2}. Examples to which Corollary~\ref{intro cor howson} applies include limit groups (such as free and surface groups, Example~\ref{ex limit}) and locally quasiconvex hyperbolic groups (Example~\ref{ex locally quasiconvex}). However, limit groups are LERF (aka subgroup separable), and locally quasiconvex hyperbolic groups are LERF, conditionally on the residual finiteness of hyperbolic groups, which for these examples leads to a more elementary proof of a stronger fact; see Lemma~\ref{lem:lerf}. One non-LERF example satisfying the hypotheses of Corollary~\ref{intro cor howson} is the solvable Baumslag--Solitar group $BS(1, n), n > 1$, however in this case the conclusion can be obtained more directly; see Example~\ref{ex:bs}.

\medskip

\textbf{Actions on flag complexes.}
One can view an action on an (infinite) set $S$ as an action on the flag complex that is the ``full simplex'' on $S$, and the former is oligomorphic if and only if the latter has finitely many orbits of simplices per dimension. This can be generalised by considering an arbitrary action on an infinite-dimensional flag complex, and instead of permutational wreath products looking at graph-wreath products \cite{kropholler:martino}. In this direction, we prove the following:

\begin{theorem}
\label{intro thm flag}
    Let $G$ be a group in $\HF$. Suppose that $G$ acts simplicially on an infinite-dimensional flag complex $X$, cocompactly on each skeleton. Then either $G$ is not of type $\F_\infty$, or there exists a simplex $\sigma$ whose stabiliser in $G$ is not of type $\F_\infty$.
\end{theorem}

\begin{theorem}
\label{intro thm graph-wreath}
    Let $G$ be a group in $\HF$. Suppose that $G$ acts simplicially on an infinite-dimensional flag complex $X$. If $B$ is an infinite group, then the graph-wreath product $B \grwr_X G$ is not of type $\F_\infty$.
\end{theorem}

If $G$ additionally satisfies the conditions of Corollary~\ref{intro cor howson} then we get the analogous conclusions in the flag complex and graph-wreath product realm (Corollary~\ref{flag howson}). Note that here the homological finiteness properties from the previous results are replaced with homotopical finiteness properties. Recall that a group $G$ is of \emph{type $\F_n$} if it admits a classifying space with finite $n$-skeleton, and it is \emph{type $\F_\infty$} if it is of type $\F_n$ for all $n$. This is because we rely on results from \cite{kropholler:martino}, which are only proved for the homotopical finiteness properties. As the authors do in that paper, we conjecture that the homological version of their results holds, and then a homological version of Theorems~\ref{intro thm flag} and~\ref{intro thm graph-wreath} would follow.

\medskip

\textbf{Connections to the Boone--Higman conjecture.} The \emph{Boone--Higman conjecture} predicts that every finitely generated group with solvable word problem embeds in a finitely presented simple group \cite{boone74}. Twisted Brin--Thompson groups have recently been used extensively to prove instances of the Boone--Higman conjecture \cite{bbmz_survey}. They are convenient targets into which to try and embed groups, since they are always simple \cite{belk22}, and there is a precise characterisation of when they are finitely presented \cite{zaremsky_fp_tbt}. Since $G$ embeds in $SV_G$, using twisted Brin--Thompson groups to prove the Boone--Higman conjecture for a given group often amounts to embedding in a group $G$ acting nicely on a set $S$.

In particular, Belk, Fournier-Facio, Hyde and Zaremsky recently proved that the group $\Aut(F_n)$ embeds into a group $G$ admitting an action on a set $S$ such that the twisted Brin--Thompson group $SV_G$ is finitely presented, so $\Aut(F_n)$ satisfies the Boone--Higman conjecture \cite{BFFHZ}. More precisely, $G = \Aut_\Gamma(\Gamma * F_n)$, where $\Gamma$ is a finitely presented simple mixed identity-free group, and $S = \Hom_\Gamma(\Gamma * F_n, \Gamma)$, with the action by precomposition. See Section~\ref{sec:bhc} for more explanation of the notation and terminology.

Thanks to our results, we obtain the following surprising corollary.

\begin{corollary}
\label{Burger Mozes}
    With the notation above, suppose that $\Gamma$ is a simple Burger--Mozes group. Then $SV_G$ is finitely presented, but not of type $\FP_\infty$.
\end{corollary}

This gives a negative answer to the ``natural conjecture'' raised in \cite[Remark~3.7]{BFFHZ}, and raises the question of which finiteness properties $SV_G$ has. For example, is it of type $\FP_3$? This result is surprising since $\Gamma$ and $G$ here are of type $\FP_\infty$, and even (virtually) of type $\F$ (Proposition~\ref{aut type F}).

\subsection*{Acknowledgments} FFF is supported by the Herchel Smith Postdoctoral Fellowship Fund. The authors thank Laurent Bartholdi, Corentin Bodart, Marc Burger, Elia Fioravanti, Marco Linton and Claudio Llosa Isenrich for useful discussions, and the anonymous referee for useful comments.

\section{Oligomorphic actions}\label{sec:olig}

In this section we prove our main results, Theorems~\ref{intro thm action} and~\ref{intro thm constructions}. We also prove Corollary~\ref{intro cor howson}. The bulk of the work is the statement about permutational wreath products, so we begin with that.

\begin{proof}[Proof of Theorem~\ref{intro thm constructions}]
    Let $G$ be a group in $\HF$ acting on a set $S$, and let $B$ be an infinite group. Suppose that $B \wr_S G$ is of type $\FP_\infty$: we will show that $S$ must be finite. First, since $B \wr_S G$ is assumed to be $\FP_\infty$, it is countable and therefore $S$ is countable. Moreover, by Corollary~\ref{wreath FPinf}, since $B \wr_S G$ is $\FP_\infty$, so is $\Z \wr_S G$, so from now on we may assume that $B = \Z$.

    The permutational wreath product fits into a short exact sequence
    \[1 \to \bigoplus_S \Z \to \Z \wr_S G \to G \to 1.\]
    Because $\Z^k$ is in $\HF$ and $\HF$ is closed under countable directed unions \cite[2.2(ii)]{kropholler:HF}, the direct sum $\bigoplus_S \Z$ is in $\HF$. Because $\HF$ is closed under extensions \cite[2.3]{kropholler:HF}, we conclude that $\Z \wr_S G$ is in $\HF$. By \cite[Theorem A]{kropholler:mislin}, we deduce that $\Z \wr_S G$ admits a finite-dimensional classifying space for proper actions. Since the subgroup $\bigoplus_S \Z$ is torsion-free, this implies that it acts freely on a finite-dimensional contractible complex, and therefore has finite cohomological dimension. This is only possible if $S$ is finite.

    Finally, by \cite[Section 2]{zaremsky_fp_tbt}, if $SV_G$ is of type $\FP_\infty$ then so is $\Z \wr_S G$, so the result about $SV_G$ is immediate.
\end{proof}

\begin{remark}
\label{rem wreath like}
    Note that in the proof above, we never used that the short exact sequence defining $\Z \wr_S G$ splits. Therefore the same argument shows that for $G\in\HF$ and $S$ infinite, no \emph{wreath-like product} $\Gamma \in \mathcal{WR}(\Z, G \curvearrowright S)$ can be of type $\FP_\infty$ (see \cite{wreathlike} for the definition). For a general infinite base $B$, we cannot reach the same conclusion, because we reduced to $\Z$ via the results in the appendix, which are specific to permutational wreath products. However it is easy to see that the same argument allows us to cover the case in which $B$ contains a copy of $\Z$.
\end{remark}

Now we quickly obtain our main result.

\begin{proof}[Proof of Theorem~\ref{intro thm action}]
    By Theorem~\ref{intro thm constructions}, if $G$ is a group in $\HF$ acting on an infinite set $S$, then $\Z \wr_S G$ is not of type $\FP_\infty$. By Corollary~\ref{wreath FPinf}, either the action is not oligomorphic, or $G$ is not of type $\FP_\infty$, or there exists a non-empty finite subset of $S$ whose stabiliser in $G$ is not of type $\FP_\infty$. The first case is excluded by assumption. Lemma~\ref{lem:converse} below shows that the second case implies the third case.
\end{proof}

We can also strengthen Theorem~\ref{intro thm action} by obtaining more control on the non-empty subsets whose stabilisers are not of type $\FP_\infty$.

\begin{corollary}
\label{cor large sets}
    Let $G$ be a group in $\HF$ acting on an infinite set $S$. If the action is oligomorphic, then for every $k \geq 1$, there exists a subset of $S$ of size at least $k$ whose stabiliser is not of type $\FP_\infty$.
\end{corollary}

\begin{proof}
    Let $\mathcal{P}_k(S)$ denote the set of subsets of $S$ of size $k$. Since the action of $G$ on $S$ is oligomorphic, so is the action of $G$ on the infinite set $\mathcal{P}_k(S)$. Theorem~\ref{intro thm action} gives a non-empty finite subset of $\mathcal{P}_k(S)$ whose stabiliser in $G$ is not of type $\FP_\infty$. This is equal to the stabiliser of a subset of $S$ of size at least $k$.
\end{proof}

\subsection{A question}
\label{subsec question}

One reason that Theorem~\ref{intro thm action} is surprising is that it is in some sense a negation of the converse of the following easy fact.

\begin{lemma}
\label{lem:converse}
    Let $G$ be a group acting oligomorphically on a set $S$. If the stabiliser in $G$ of every non-empty finite subset of $S$ is of type $\FP_\infty$, then $G$ is of type $\FP_\infty$.
\end{lemma}

\begin{proof}
    Let $X$ be the simplicial complex with vertex set $S$ such that every finite subset of $S$ spans a simplex. The oligomorphic action of $G$ on $S$ extends to an action of $G$ on $X$ with finitely many orbits of simplices per dimension. The complex $X$ is contractible, so its $n$-skeleton $X^{(n)}$ is $(n-1)$-connected, for each $n$. The action of $G$ on $X^{(n)}$ is cocompact, and all simplex stabilisers are of type $\FP_\infty$, so $G$ is of type $\FP_n$ for example by \cite[Proposition~1.1]{brown87}. This works for all $n$, so $G$ is of type $\FP_\infty$.
\end{proof}

The most basic example of a group for which all the hypotheses of Lemma~\ref{lem:converse} apply is Thompson's group $F$: it is of type $\FP_\infty$ \cite{browngeo84}, it acts oligomorphically on $\Z\left[\frac{1}{2}\right] \cap (0, 1)$ \cite[Lemma 4.2]{cannon96}, and the stabiliser of any $n$-tuple is isomorphic to $F^{n+1}$ \cite[Lemma 4.4]{cannon96}, hence is of type $\FP_\infty$. In fact, Thompson's group $F$ is the prototypical example of a group not in $\HF$ \cite[End of Section 1]{kropholler:HF}. Similarly, many ``Thompson-like'' groups admit an action as in Lemma~\ref{lem:converse}. It would be interesting to find examples of a different nature:

\begin{question}
\label{question}
    Does there exist a group $G$ satisfying the hypotheses of Lemma~\ref{lem:converse} (with $S$ infinite) that does not contain a copy of Thompson's group $F$?
\end{question}

By Theorem~\ref{intro thm action}, such a group cannot be in $\HF$. We only know of two sources of groups not in $\HF$ that do not contain a copy of Thompson's group $F$. One source is groups with strong fixed point properties coming from small cancellation constructions \cite{fixedpoint1, fixedpoint2, fixedpoint3}. Such groups are, essentially by construction, not even of type $\FP_2$ \cite[Remark 4.5]{fixedpoint2}.

Another source is groups with \emph{jump rational cohomology}, that is, groups $G$ with $\cdq(G) = \infty$, and with the property that there exists $n \in \N$ such that every $H \leq G$ with $\cdq(H) < \infty$ satisfies $\cdq(H) \leq n$. Such groups cannot be in $\HF$ \cite[Theorem 3.2]{jumps}. Besides some overlap with the previous source \cite{fixedpoint2}, examples include some branch groups such as Grigorchuk's group \cite[Theorem 4.11]{gandini}; see \cite[Proposition 3.22]{subgroupinduction} for more examples. However, branch groups cannot act oligomorphically, or even almost-$2$-transitively \cite{francoeur}.

Let us remark that a group answering Question~\ref{question} would necessarily have infinite cohomological dimension (otherwise it would be in $\HF$). Therefore, a \emph{torsion-free, finitely presented} example would answer a question of Stefan Witzel \cite{witzel:question}, which asks about the existence of a torsion-free group of type $\F_\infty$ with infinite cohomological dimension that does not contain a copy of Thompson's group $F$.

\subsection{Stronger hypotheses, stronger conclusions}

Now let us prove the stronger result under the additional assumptions of the Howson property and homological supercoherence.

\begin{proof}[Proof of Corollary~\ref{intro cor howson}]
    Let $G$ be a finitely generated homologically supercoherent group $G$ in $\HF$ with the Howson property, acting oligomorphically on an infinite set $S$. Suppose that every element of $S$ has a finitely generated stabiliser in $G$. By the Howson property, every finite subset of $S$ has a finitely generated stabiliser in $G$. By homological supercoherence, $G$ and all of these stabilisers are of type $\FP_\infty$, contradicting Theorem~\ref{intro thm action}.
\end{proof}

In fact the same argument as in Corollary~\ref{cor large sets} strengthens this to:

\begin{corollary}
\label{cor large sets howson}
    Let $G$ be a group in $\HF$ acting on an infinite set $S$. Suppose $G$ is finitely generated, homologically supercoherent, and has the Howson property. If the action is oligomorphic, then for every $k \geq 1$, there exists a subset of $S$ of size $k$ whose stabiliser is not finitely generated. \qed
\end{corollary}

We also get a sort of analog of Theorem~\ref{intro thm constructions}, where the stronger assumptions on $G$ lead to the stronger conclusion that the permutational wreath product and twisted Brin--Thompson group cannot even be of type $\FP_2$. (Here we need to assume that the action is oligomorphic, unlike in Theorem~\ref{intro thm constructions}, but need to make no assumptions about $B$.)

\begin{corollary}\label{cor:howson_FP2}
    Let $G$ be a finitely generated homologically supercoherent group in $\HF$ with the Howson property, acting oligomorphically on a set $S$. If $B$ is a non-trivial group, then the permutational wreath product $B \wr_S G$ is not of type $\FP_2$. Moreover, the twisted Brin--Thompson group $SV_G$ (which exists only when the action is faithful) is not of type $\FP_2$.
\end{corollary}

\begin{proof}
    As in the proof of Theorem~\ref{intro thm constructions}, if $SV_G$ is of type $\FP_2$ then so is $\Z \wr_S G$, so we focus on the statement on wreath products. Let $B$ be a non-trivial group and suppose that $B\wr_S G$ is of type $\FP_2$. By \cite[Lemma~5.1]{bartholdi15}, every element of $S$ consequently has a finitely generated stabiliser in $G$. This contradicts Corollary~\ref{intro cor howson}, and we are done.
\end{proof}

There is another sufficient condition for getting the conclusions of Corollary~\ref{intro cor howson}, namely the property of $G$ being LERF. This essentially already existed in the literature, due to Cornulier \cite{cornulier06}. Recall that a group is \emph{LERF (locally extended residually finite)} if every finitely generated subgroup is closed in the profinite topology, i.e., it is the intersection of a collection of finite index subgroups. For this result we do not need the action to be oligomorphic, but only \emph{almost 2-transitive}, meaning there are finitely many orbits of 2-element subsets.

\begin{lemma}\label{lem:lerf}
    Let $G$ be a group acting almost 2-transitively on an infinite set $S$. If $G$ is LERF, then there exists a point in $S$ whose stabiliser is not finitely generated.
\end{lemma}

\begin{proof}
    Up to restricting to a single orbit, we may assume the action is transitive, and so the goal is to show that no point stabilisers are finitely generated. Let $s\in S$ and let $H=\Stab_G(s)$. Since the action is almost 2-transitive, there are finitely many double cosets $HgH$ for $g\in G$, i.e., $H$ has finite biindex in $G$. By \cite[Lemma~3.2]{cornulier06}, this implies that the profinite closure of $H$ in $G$ (i.e., the intersection of all finite index subgroups containing $H$) has finite index. Now if $H$ were finitely generated, the LERF property would imply that $H$ has finite index in $G$, which is impossible since $S$ is infinite and $G$ acts transitively on $S$.
\end{proof}

\begin{example}
\label{ex limit}
    Recall that a \emph{limit group} is a finitely generated fully residually free group, that is, one for which every finite subset maps injectively into a free quotient. It follows from the constructive definition of limit groups \cite[Theorem 4.6]{limit} (and the fact that abelian subgroups of limit groups are finitely generated \cite[Theorem 3.9]{limit}) that limit groups are of type $\F$, i.e., they admit finite classifying spaces. In particular, they are of type $\FP_\infty$ and belong to $\HF$. Moreover, finitely generated subgroups of limit groups are limit groups, from which it follows that limit groups are homologically supercoherent. Finally, they have the Howson property \cite{limit:howson}. However, they are LERF \cite{limit:lerf}.
\end{example}

\begin{example}
\label{ex locally quasiconvex}
    A hyperbolic group is \emph{locally quasiconvex} if every finitely generated subgroup is quasiconvex. Recall that hyperbolic groups admit a finite model for the classifying space for proper actions \cite{meintrup02}, therefore they are of type $\FP_\infty$ and in $\HF$. Quasiconvex subgroups of hyperbolic groups are themselves hyperbolic, and this implies that locally quasiconvex hyperbolic groups are homologically supercoherent. Moreover, the intersection of two quasiconvex subgroups is quasiconvex, and so locally quasiconvex hyperbolic groups have the Howson property. However, locally quasiconvex hyperbolic groups are LERF, conditionally on the residual finiteness of all hyperbolic groups, which is a long standing question \cite[5.3.B Remark]{gromov}. More generally, if all hyperbolic groups are residually finite, then all quasiconvex subgroups of hyperbolic groups are closed in the profinite topology \cite{qcerf}.
\end{example}

Note that both these examples are either virtually abelian or acylindrically hyperbolic. It is known that acylindrically hyperbolic groups admit highly transitive (hence oligomorphic) actions with finite kernel \cite{hull16}.

We end with one non-LERF example satisfying the hypotheses of Corollary~\ref{intro cor howson}, where however the conclusion can be obtained more easily.

\begin{example}
\label{ex:bs}
    Let $n > 1$ and let $BS(1, n) = \langle a, t \mid tat^{-1} = a^n \rangle$ denote the solvable Baumslag--Solitar group. It lies in $\HF$, for instance because it has cohomological dimension $2$. It is not LERF, indeed the subgroup $\langle a \rangle$ is not separable \cite{bs:lerf}. Every finitely generated subgroup is either cyclic or has finite index \cite[Theorem~2]{bs:subgroups}. This implies at once that $BS(1, n)$ is homologically supercoherent and has the Howson property.
    
    However, the conclusion of Corollary~\ref{intro cor howson} is easier to obtain. Indeed, suppose that $BS(1, n)$ acts on an infinite set $S$ oligomorphically. Then some point stabiliser must be of infinite index, and act oligomorphically on the complement of the point. Because $\Z$ cannot act oligomorphically on an infinite set, the stabiliser must be infinitely generated, again by \cite[Theorem~2]{bs:subgroups}.
\end{example}

We refer the reader to \cite{howson:gog2,howson:gog1} for recent results on the Howson property for graphs of groups, which are hence potential applications of Corollary~\ref{intro cor howson}.

\section{Flag complexes and graph-wreath products}

Recall that a \emph{flag complex} is a simplicial complex in which every finite complete subgraph of the 1-skeleton lies in a simplex. A flag complex is completely determined by its 1-skeleton, so when talking about flag complexes we are often really talking about simplicial graphs. Given a simplicial graph $\Gamma$ and a group $B$, the \emph{graph product} $B^\Gamma$ is the group generated by a copy of $B$ for each vertex of $\Gamma$, modulo the relation that whenever two vertices are adjacent, their corresponding copies of $B$ commute. For example if $\Gamma$ has no edges then $B^\Gamma$ is the free product of a copy of $B$ for each vertex, and if every pair of vertices of $\Gamma$ spans an edge then $B^\Gamma$ is the direct sum $\bigoplus_{\Gamma^{(0)}}B$. For $X$ a flag complex, write $B^X$ for $B^{X^{(1)}}$.

If a group $G$ acts simplicially on a flag complex $X$, then $G$ acts on the graph product $B^X$ by permuting the copies of $B$. This leads us to the following generalisation of permutational wreath products \cite{kropholler:martino}.

\begin{definition}[Graph-wreath product]
    Let $G$ be a group acting simplicially on a flag complex $X$. Let $B$ be a group. The \emph{graph-wreath product} $B \grwr_X G$ is the semidirect product
    \[
    B^X \rtimes G,
    \]
    where $G$ acts on $B^X$ as above.
\end{definition}

For example if $X^{(1)}$ is complete then $B \grwr_X G = B \wr_{X^{(0)}} G$.

\begin{proof}[Proof of Theorem~\ref{intro thm graph-wreath}]
    Let $G$ be a group in $\HF$ acting on a flag complex $X$ and let $B$ be an infinite group. Suppose that $B \grwr_X G$ is of type $\F_\infty$: we will show that $X$ is finite-dimensional. First, $B \grwr_X G$ is countable and therefore $X^{(0)}$ is countable. Moreover, Theorem~\ref{graph-wreath Fn} implies that $\Z \grwr_X G$ is also of type $\F_\infty$, so from now on we may assume that $B = \Z$.

    The graph-wreath product fits into a short exact sequence
    \[1 \to \Z^X \to \Z \grwr_X G \to G \to 1.\]
    Now $\Z^X$ is a directed union of right angled Artin groups, which are linear \cite{raags:linear}, hence in $\HF$, so by closure under countable directed unions \cite[2.2(ii)]{kropholler:HF}, we see that $\Z^X$ is in $\HF$. Because $\HF$ is closed under extensions \cite[2.3]{kropholler:HF}, we conclude that $\Z \grwr_X G$ is in $\HF$. By \cite[Theorem A]{kropholler:mislin}, we deduce that $\Z \grwr_X G$ admits a finite-dimensional classifying space for proper actions. Since the subgroup $\Z^X$ is torsion-free, this implies that it acts freely on a finite-dimensional contractible complex, and therefore has finite cohomological dimension. But if $X$ contains a $k$-simplex, then $\Z^X$ contains a copy of $\Z^{k+1}$, so this forces $X$ to be finite-dimensional.
\end{proof}

\begin{proof}[Proof of Theorem~\ref{intro thm flag}]
    By Theorem~\ref{intro thm graph-wreath}, if $G$ is a group in $\HF$ acting on an infinite-dimensional flag complex $X$, then $\Z \grwr_X G$ is not of type $\F_\infty$. By Theorem~\ref{graph-wreath Fn}, either the action is not cocompact on some skeleton, or $G$ is not of type $\F_\infty$, or the stabiliser of some simplex is not of type $\F_\infty$, as desired.
\end{proof}

Now the same argument as in the proof of Corollary~\ref{intro cor howson} gives:

\begin{corollary}
\label{flag howson}
    Let $G$ be a group in $\HF$ acting on an infinite-dimensional flag complex $X$. Suppose that $G$ is finitely generated, supercoherent, and has the Howson property. If the action is cocompact on each skeleton, then there exists a vertex in $X$ whose stabiliser is not finitely generated. \qed
\end{corollary}

Here a group is said to be \emph{supercoherent} if every finitely generated subgroup is of type $\F_\infty$.

\begin{remark}
    We are not aware of a natural analog of twisted Brin--Thompson groups using actions on flag complexes, beyond the case where every pair of vertices is adjacent. The first question in this direction is whether the pipeline from $V\times V$ to $2V$ \cite{brin04} has some variation that instead begins with $V*V$.
\end{remark}

\section{Connections to the Boone--Higman conjecture}\label{sec:bhc}

We start with an appropriate definition of Burger--Mozes groups, which serves our purpose.
Let us quickly recall some basic notions. Let $\Omega$ be a finite set, and let $\Sigma$ be a group of permutations of $\Omega$. We say that $\Sigma$ is \emph{primitive} if the only partitions of $\Omega$ that are preserved by $\Sigma$ are the trivial ones (i.e., $\{\Omega\}$ and the partition into singletons).
Now let $T$ be a locally finite tree, and let $H < \Aut(T)$. For a vertex $x$ of $T$, the stabiliser $G_x$ acts as a finite permutation group on the set of edges adjacent to $x$. We say that $G$ is \emph{locally primitive} if this finite permutation group is primitive, for every $x$.

\begin{definition}[Burger--Mozes group]
\label{def:BM}
    Let $T_1, T_2$ be regular trees of finite degree at least $3$. A \emph{Burger--Mozes group} is a cocompact lattice $\Gamma < \Aut(T_1) \times \Aut(T_2)$ such that $\overline{\mathrm{pr}_i(\Gamma)} < \Aut(T_i)$ is locally primitive for $i=1,2$.
\end{definition}

Burger--Mozes groups are finitely presented and infinite. Moreover, there exist \emph{simple} Burger--Mozes groups \cite{burgermozes}.

\begin{proposition}[{\cite[Corollary 1.4.3]{BMZ}}]
\label{prop:out:finite}
Let $\Gamma$ be a Burger--Mozes group. Then $\Out(\Gamma)$ is finite.
\end{proposition}

Now let $\Gamma$ be a simple Burger--Mozes group. We consider the group $G = \Aut_\Gamma(\Gamma*F_n)$ of automorphisms of the free product $\Gamma*F_n$ that restrict to the identity on $\Gamma$, where $F_n$ is a free group of rank $n \geq 2$. Moreover, we let $S = \Hom_\Gamma(\Gamma*F_n, \Gamma)$ be the set of homomorphisms from $\Gamma*F_n$ to $\Gamma$ that restrict to the identity on $\Gamma$. The group $G$ naturally acts on $S$ by precomposition.

\begin{proposition}[\cite{BFFHZ}]
\label{prop action properties}
    Let $\Gamma$ be a simple Burger--Mozes group, and let $G$ and $S$ be as above. Then $G$ is finitely presented, the action of $G$ on $S$ is faithful and oligomorphic, and point stabilisers are finitely generated.
\end{proposition}

\begin{proof}
    $G$ is finitely presented because $\Gamma$ is \cite[Proposition 1.1]{BFFHZ}. $\Gamma$ is mixed identity-free \cite[Remark 3.5]{BFFHZ}, which implies that the action is faithful \cite[Lemma 2.2]{BFFHZ}. The action is oligomorphic, and even highly transitive, because $\Gamma$ is simple \cite[Proposition 2.3]{BFFHZ}. Finally, point stabilisers are finitely generated by \cite[Proposition 2.6]{BFFHZ}.
\end{proof}

What distinguishes Burger--Mozes groups from other finitely presented simple groups considered in \cite{BFFHZ} is the following:

\begin{proposition}
\label{aut type F}
    Let $\Gamma$ be a simple Burger--Mozes group, and let $G$ be as above. Then $G$ is virtually of type $\F$. In particular, $G$ belongs to $\HF$, and it is of type $\FP_\infty$.
\end{proposition}

\begin{proof}
    By Grushko's Theorem, every automorphism of $\Gamma*F_n$ sends $\Gamma$ to a conjugate of itself. Composing by an inner automorphism yields an automorphism of $\Gamma$. This defines a homomorphism $\Out(\Gamma*F_n) \to \Out(\Gamma)$. The kernel coincides with the image of $G$ inside $\Out(\Gamma*F_n)$. Since $G$ intersects $\Inn(\Gamma*F_n)$ trivially, by Proposition~\ref{prop:out:finite} we can now identify $G$ with a finite index subgroup of $\Out(\Gamma*F_n)$.

    Therefore it remains to see that $\Out(\Gamma*F_n)$ is virtually of type $\F$. This follows from a theorem of Guirardel and Levitt \cite[Theorem 5.2]{GLouter}, which applies because $\Gamma \cong \Inn(\Gamma)$ is freely indecomposable, of type $\F$, and $\Out(\Gamma)$ is finite (Proposition~\ref{prop:out:finite}).
\end{proof}

By contrast, if $\Gamma$ is a Thompson-like finitely presented simple group, then it contains Thompson's group $F$, which is not in $\HF$ \cite[End of Section 1]{kropholler:HF}, and hence $G=\Aut_\Gamma(\Gamma*F_n)$ is not in $\HF$.

\begin{proof}[Proof of Corollary~\ref{Burger Mozes}]
    Let $\Gamma$ be a simple Burger--Mozes group, and let $G = \Aut_\Gamma(\Gamma*F_n)$ act on $S = \Hom_\Gamma(\Gamma*F_n, \Gamma)$ by precomposition. The twisted Brin--Thompson group $SV_G$ exists because the action is faithful (Proposition~\ref{prop action properties}). Proposition~\ref{prop action properties} gives all of the hypotheses necessary for the criterion in \cite[Theorem A]{zaremsky_fp_tbt}, which shows that $SV_G$ is finitely presented. However, $SV_G$ is not of type $\FP_\infty$: this follows from the second item of Theorem~\ref{intro thm constructions}, which applies because $G$ is in $\HF$, by Proposition~\ref{aut type F}.
\end{proof}

\appendix

\section{Finiteness properties of (graph-)wreath products}
\label{appendix}

In this appendix, we recall and slightly improve on the criteria for the finiteness properties of permutational wreath products from \cite{bartholdi15}, and of graph-wreath products from \cite{kropholler:martino}: see Theorems~\ref{wreath FPn} and~\ref{graph-wreath Fn}. 

The key insight is to make use of certain quasi-retractions of groups. Recall that a \emph{quasi-retraction} $\rho\colon X\to Y$ of metric spaces is a coarse Lipschitz map that admits a \emph{quasi-section}, meaning a coarse Lipschitz map $\iota\colon Y\to X$ such that $\rho\circ \iota$ is uniformly close to the identity on $Y$. Call $Y$ a \emph{quasi-retract} of $X$ if there exists a quasi-retraction $X\to Y$. If moreover $\rho\circ \iota$ equals the identity, we will refer to $\rho$ as a quasi-retraction \emph{with an honest section}, and refer to $Y$ as a quasi-retract of $X$ \emph{with an honest section}. If $G$ is a group of type $\FP_n$ (respectively $\F_n$) with a word metric coming from a finite generating set, then every quasi-retract of $G$ is also of type $\FP_n$ (respectively $\F_n$) \cite{alonso94}.

Our improvement to the results in \cite{bartholdi15} and \cite{kropholler:martino} is to replace the hypothesis that the base group has infinite abelianisation by the hypothesis that the base group is infinite. This is made possible by the following observation, due to Corentin Bodart \cite{corentin} (see also \cite[Lemma 3.2]{corentin:citation} for a similar statement).

\begin{proposition}
\label{corentin}
    Let $B$ be a finitely generated infinite group. Then there exists a quasi-retraction $B \to \Z$ with an honest section.
\end{proposition}

It is important here that the composition $\Z \to B \to \Z$ is the identity, and not just close to the identity. This stronger property will be needed in the proof of Corollary~\ref{retraction wreath}: see Remark~\ref{remark quasiretraction section}.

\begin{proof}
    Let $(x_i)_{i \in \Z}$ be a bi-infinite geodesic in a fixed Cayley graph of $B$, with $x_0 = 1$. Define $\rho \colon B \to \Z$ by
    \[
    \rho(b) = \lim\limits_{i \to \infty} \left(d_B(x_i, 1) - d_B(x_i, b)\right) = \lim\limits_{i \to \infty} \left(i - d_B(x_i, b)\right).
    \]
    The limit exists: by the triangle inequality we have
    $d_B(x_{i+1}, b) \leq d_B(x_i, b) + d(x_i, x_{i+1})$, so $d_B(x_i, b) - d_B(x_{i+1}, b) \geq -1$. Hence the sequence $(i - d_B(x_i, b))_i$ is non-decreasing. Since it is also bounded above by $d_B(1,b)$ (again by the triangle inequality), the sequence converges. One observes that $\rho$ is $1$-Lipschitz, and admits an honest section $\iota \colon \Z \to B$ defined by $\iota(i) = x_i$.
\end{proof}

Soon we will need to directly manipulate elements of graph-wreath products, so let us pin down some precise details. An element of $B^X$ is a function $f\colon X^{(0)}\to B$ satisfying $f(v)=1$ for all but finitely many $v\in X^{(0)}$. The (left) action of $G$ on $B^X$ has $g\in G$ send $f\in B^X$ to $f\circ g^{-1}$. For $b\in B$ and $v\in X^{(0)}$, let us write $b_v \in B^X$ for the element defined by $b_v(v)=b$ and $b_v(v')=1$ for all $v'\ne v$. Note that under the action of $G$ on $B^X$, $g$ sends $b_v$ to $b_v\circ g^{-1}=b_{g(v)}$. If $A\subseteq B$ write $A_v=\{a_v\mid a\in A\}$. Write elements of $B \grwr_X G$ as ordered pairs $(f,g)$ for $f\in B^X$ and $g\in G$, so the group operation is given by $(f,g)(f',g')=(f (f'\circ g^{-1}) , gg')$. Note that $B\grwr_X G$ is generated by $G$ together with the elements $(b_v,1)$ for $b\in B$ and $v$ coming from a set of representatives of $G$-orbits in $X^{(0)}$.

\begin{corollary}
\label{retraction wreath}
    Suppose that $B$ and $G$ are finitely generated groups, with $B$ infinite, and $X$ is a flag complex on which $G$ acts simplicially and with finitely many orbits of vertices. Then there exists a quasi-retraction $B \grwr_X G \to \Z \grwr_X G$ with an honest section.

    In particular, if $S$ is a set on which $G$ acts with finitely many orbits, then there exists a quasi-retraction $B \wr_S G \to \Z \wr_S G$ with an honest section.
\end{corollary}

\begin{proof}
    The case of permutational wreath products follows from that of graph-wreath products, by taking $X$ to be the full simplex on $S$. Our hypotheses ensure that $W = B \grwr_X G$ is finitely generated \cite[Lemma 2.5]{kropholler:martino}, and a finite generating set $A$ can be built as follows: fix a finite generating set $C_B$ for $B$, a finite generating set $C_G$ for $G$, and a finite set $V$ of orbit representatives for the action of $G$ on the vertex set of $X$. The generating set $A$ of $W$ now consists of all $(c_v,1)$ for $c\in C_B$ and $v\in V$ together with all $(1,c)$ for $c\in C_G$. Similarly, denoting $\Z = \langle t \rangle$ multiplicatively, the group $\Z \grwr_X G$ is generated by a finite set $A'$ consisting of the elements $(t_v, 1)$ for $v \in V$, together with $(1, c)$ for $c \in C_G$.

    Let $\rho \colon B \to \Z$ be the quasi-retraction from Proposition~\ref{corentin}. This naturally extends to a map $\sigma \colon B \grwr_X G \to \Z \grwr_X G$ via
    \[
    \sigma(f,g) \coloneqq (\rho\circ f,g).
    \]
    We claim that, when $B \grwr_X G$ is endowed with the generating set $A$, and $\Z \grwr_X G$ is endowed with the generating set $A'$, the map $\sigma$ is a quasi-retraction with an honest section. Since $\rho\colon B\to\Z$ has an honest section, so does $\sigma$. Thus we are left to show that $\sigma$ is coarse Lipschitz, i.e., that there is a uniform bound on the distance in $\Z \grwr_X G$ between $\sigma(wa)$ and $\sigma(w)$, for all $w\in W$ and $a\in A$. Say $w=(f,g)$, and first suppose $a=(1,c)$ for $c\in C_G$. Then
    \[
    \sigma(wa) = \sigma((f,g)(1,c)) = \sigma(f,gc) = (\rho\circ f,gc)=(\rho\circ f,g)(1,c) = \sigma(w)(1,c).
    \]
    Now suppose $a=(c_v,1)$ for $c\in C_B$ and $v\in V$. Then
    \begin{align*}
    \sigma(wa) &= \sigma((f,g)(c_v,1)) = \sigma(f c_{g(v)},g) \\
    &= (\rho\circ(f c_{g(v)}),g) = ((\rho\circ f)z_{g(v)},g) \\
    &= (\rho\circ f,g)(z_v,1) = \sigma(w)(z_v,1),
    \end{align*}
    where $z=(\rho(f(g(v))))^{-1}\rho(f(g(v))c) \in \Z$. Since $\rho$ is coarse Lipschitz, $z$ is at uniformly bounded distance to 1, and we conclude that $\sigma(wa)$ is at uniformly bounded distance to $\sigma(w)$.
\end{proof}

\begin{remark}
\label{remark quasiretraction section}
    If $\rho \colon B \to \Z$ were simply a quasi-retraction, then $\sigma$ might not be a quasi-retraction, because the finite errors of $\rho$ can add up along the direct sum. For a concrete example, consider the embedding $\rho \colon \{1\} \to \Z/2$, which is a quasi-retraction (that does not admit a section). Then the corresponding map $\sigma \colon \{1\} \wr \Z \to \Z/2 \wr \Z$ is not a quasi-retraction: indeed $\{1\} \wr \Z \cong \Z$, and every quasi-retract of $\Z$ is finitely presented, unlike $\Z/2 \wr \Z$.
\end{remark}

We also include another observation, which does not improve on the finiteness properties but does streamline a step of the proofs.

\begin{lemma}
\label{retraction base}
    Suppose that $B$ and $G$ are finitely generated groups, and $X$ is a flag complex on which $G$ acts simplicially and with finitely many orbits of vertices. Then there exists a quasi-retraction $B \grwr_X G \to B$. In particular, if $S$ is a set on which $G$ acts with finitely many orbits, then there exists a quasi-retraction $B \wr_S G \to B$.
\end{lemma}

We will use this in conjunction with \cite{alonso94}, to show that $B$ inherits the finiteness properties of $B \grwr_X G$. This improves on \cite[Proposition 5.2]{bartholdi15}, which reaches the same conclusion under stronger hypotheses.

\begin{proof}
    Fix some $u\in X^{(0)}$. Let $\beta\colon B \grwr_X G \to B$ be the function $\beta(f,g)=f(u)$. Keeping the same notation as in the proof of Corollary~\ref{retraction wreath}, we endow $B$ with the generating set $C_B$ and $B \grwr_X G$ with the generating set $A$, and claim that $\beta$ is a quasi-retraction. Clearly $\beta$ has an honest section $b\mapsto (b_u,1)$. Now let $w \in W=B\grwr_X G$ and $a \in A$; we need to prove that $\beta(wa)$ is uniformly close to $\beta(w)$. First suppose $a=(1,c)$ for $c\in C_B$. Then
    \[\beta(wa)=\beta((f,g)(1,c))=\beta(f,gc)=f(u)=\beta(w).\]
    Next suppose $a=(c_v,1)$ for $c\in C_B$ and $v\in V$. Then
    \[\beta(wa) = \beta((f,g)(c_v,1)) = \beta(f c_{g(v)},g) = f(u) c_{g(v)}(u).\]
    If $u\ne g(v)$ then this equals $f(u)=\beta(f,g)$. If $u=g(v)$ then this equals $f(u)c=\beta(f,g)c$, and we are done.
\end{proof}

We can now prove the main results of the appendix. We start with permutational wreath products.

\begin{theorem}
\label{wreath FPn}
    Let $W = B \wr_S G$ be a permutational wreath product, where $B$ is infinite. Then the following are equivalent:
    \begin{itemize}
        \item $W$ is of type $\FP_n$;
        \item $B$ and $G$ are of type $\FP_n$, and for all $1 \leq i \leq n$, $G$ acts on $S^i$ with stabilisers of type $\FP_{n-i}$ and finitely many orbits.
    \end{itemize}
\end{theorem}

This is proved in \cite[Theorem A]{bartholdi15}, under the assumption that $B$ has infinite abelianisation.

\begin{proof}
    If the second bulletpoint holds, then $W$ is of type $\FP_n$, by \cite[Proposition 5.1]{bartholdi15}. Conversely, suppose that $W$ is of type $\FP_n$. Because $B$ is a quasi-retract of $W$ by Lemma~\ref{retraction base}, it is also of type $\FP_n$, by \cite{alonso94}. The statement is trivial for $n = 0$, so we assume that $n \geq 1$. Then $W$ is finitely generated, and therefore $B$ and $G$ are finitely generated and the action of $G$ on $S$ has finitely many orbits \cite[Proposition 2.1]{cornulier06}. By Corollary~\ref{retraction wreath}, there is a quasi-retraction $W \to \Z \wr_S G$, and therefore $\Z \wr_S G$ is of type $\FP_n$ by \cite{alonso94}. It now follows from \cite[Theorem A]{bartholdi15} that $G$ is of type $\FP_n$, and for all $1 \leq i \leq n$ it acts on $S^i$ with stabilisers of type $\FP_{n-i}$ and finitely many orbits.
\end{proof}

Let us also directly record the version where the bulleted statements hold for all $n$:

\begin{corollary}
\label{wreath FPinf}
    Let $W = B \wr_S G$ be a permutational wreath product, where $B$ is infinite. Then the following are equivalent:
    \begin{itemize}
        \item $W$ is of type $\FP_\infty$;
        \item $B$ and $G$ are of type $\FP_\infty$, and $G$ acts oligomorphically on $S$ with stabilisers of non-empty subsets of type $\FP_\infty$. \qed
    \end{itemize}
\end{corollary}

There is also a homotopical version.

\begin{corollary}
\label{wreath Fn}
    Let $W = B \wr_S G$ be a permutational wreath product, where $B$ is infinite. Then the following are equivalent:
    \begin{itemize}
        \item $W$ is of type $\F_n$;
        \item $B$ and $G$ are of type $\F_n$, and for all $1 \leq i \leq n$, $G$ acts on $S^i$ with stabilisers of type $\FP_{n-i}$ and finitely many orbits.
    \end{itemize}
\end{corollary}

\begin{proof}
    For $n = 1$ this is the same statement as Theorem~\ref{wreath FPn}. For $n \geq 2$, a group is of type $\F_n$ if and only if it is finitely presented and of type $\FP_n$ \cite[Section~VIII.7]{brown82}. Therefore the result is a combination of Theorem~\ref{wreath FPn} and \cite[Theorem 1.1]{cornulier06}.
\end{proof}

We obtain a similar result for graph-wreath products.

\begin{theorem}
\label{graph-wreath Fn}
    Let $W = B \grwr_X G$ be a graph-wreath product, where $B$ is infinite. Then the following are equivalent.
    \begin{itemize}
        \item $W$ is of type $\F_n$;
        \item $B$ and $G$ are of type $\F_n$, and for all $1 \leq i \leq n$, $G$ acts cocompactly on the $i$-skeleton of $X$, with the stabiliser of each $i$-simplex of $X$ of type $\FP_{n-i}$.
    \end{itemize}
\end{theorem}

\begin{proof}
    If the second bulletpoint holds, then $W$ is of type $\F_n$, by \cite[Theorem A]{kropholler:martino} (note that the third condition in the theorem is equivalent to our assumption on cocompactness and finiteness properties of the stabilisers, by \cite[Lemma 1.8]{kropholler:martino}). Conversely, suppose that $W$ is of type $\F_n$. Because $B$ is a quasi-retract of $W$ by Lemma~\ref{retraction base}, it is also of type $\F_n$, by \cite{alonso94}. The statement is trivial for $n = 0$, so we assume that $n \geq 1$. Then $W$ is finitely generated, and therefore $B$ and $G$ are finitely generated and the action of $G$ on $X$ has finitely many orbits of vertices \cite[Lemma 2.5]{kropholler:martino}. By Corollary~\ref{retraction wreath}, there is a quasi-retraction $W \to \Z \grwr_S G$, and therefore $\Z \grwr_S G$ is of type $\F_n$ by \cite{alonso94}. It now follows from \cite[Theorem C]{kropholler:martino} that $G$ is of type $\F_n$, and for all $1 \leq i \leq n$, the $i$-skeleton of $X$ is cocompact and the stabiliser of each $i$-simplex of $X$ of type $\FP_{n-i}$.
\end{proof}

Even with this improved statement, to the best of our knowledge the case when the base $B$ is finite remains open.

\bibliographystyle{alpha}
\newcommand{\etalchar}[1]{$^{#1}$}

\end{document}